\documentclass[a4paper]{amsart}

\usepackage[utf8]{inputenc}
\usepackage{amsthm,amssymb,amsmath}
\usepackage[english]{babel}

\usepackage{tikz-cd}
\usepackage{hyperref}
\usepackage{xcolor}

\newtheorem{Th}{Theorem}[section]
\newtheorem{Prop}[Th]{Proposition}
\newtheorem{Lm}[Th]{Lemma}

\theoremstyle{definition}
\newtheorem{Def}[Th]{Definition}
\newtheorem{Ques}[Th]{Question}

\newtheorem{Rem}{Remark}[section]

\title{$\mathrm{GL}_n$-structure and principal $\mathfrak{sl}_2$-triple on the cohomology ring of complex Grassmannian}
\author{Nhok Tkhai Shon Ngo}
\address{V.N.Karazin Kharkiv National University\\
	Svobody Square 4\\
	Kharkiv\\
	61022\\
	Ukraine}
\email{ngothaison17@gmail.com}

\date{\today}

\def\ad{\operatorname{ad}}
\def\tr{\operatorname{tr}}
\def\Adj{\operatorname{Adj}}

\def\Gr{\operatorname{Gr}}

\numberwithin{equation}{section}

\newcommand{\id}{\operatorname{id}}

\begin{document}
	
	\maketitle
	
	\begin{abstract}
		In this note we describe the cohomology ring of the Grassmannian of $k$-planes in $n$-dimensional complex vector space as an $\mathrm{GL}_n$-module. We give explicit formulas for the operators of its principal $\mathfrak{sl}_2$-triple. It is proved that one of these operators corresponds to the shifted cohomology degree operator and the second operator coincides with the Lefschetz map on cohomology (as in the hard Lefschetz theorem). We check that the cohomology ring of the complex Grassmannian as a $\mathrm{GL}_n$-representation is isomorphic to the $k$-th exterior power of the standard $n$-dimensional representation. 
	\end{abstract}
	
	\section{Introduction}
	
	\subsection{Cohomology ring.} 
	Fix positive integers $n$ and $k$ such that $1\le k\le n-1$. Consider the corresponding complex Grassmannian $\Gr(k,n)$ of $k$-dimensional subspaces in a $n$-dimensional complex vector space which we will denote as $W$. We will consider the cohomology ring $H^*(\Gr(k,n),\mathbb{C})$ of $\Gr(k,n)$ with complex coefficients. Since all odd cohomology groups of the $\Gr(k,n)$ are trivial its cohomology ring is commutative.

	Recall that over $\Gr(k,n)$ there are three canonical vector bundles: the \emph{trivial bundle} $T=\Gr(k,n)\times\mathbb{C}^n$ of rank $n$, the \emph{universal subbundle} $S$ of rank $k$, whose fiber at each point $\Lambda$ is the $k$-plane $\Lambda$ itself and the \emph{universal quotient bundle} $Q$ of rank $n-k$ defined as quotient $Q=S/T$.
	
	Then, for the total Chern classes $c(S)$, $c(Q)$ and $c(T)$ we have the following equalities:
	\begin{equation}
		\begin{split}
			& c(S)=c_0(S)+c_1(S)+\ldots+c_k(S),
			\\
			& c(Q)=c_0(Q)+c_1(Q)+\ldots+c_{n-k}(Q),
			\\
			& c(T)=c(S)c(Q)=1.
		\end{split}
	\end{equation}
	Note that here $c_0(S)=c_0(Q)=1$ and $c(T)=c_0(T)=1$.
	Moreover, it is known that in terms of Chern classes of bundles $S$ and $Q$ the cohomology ring $H^*(\Gr(k,n),\mathbb{C})$ can be described in the following way (see \cite[Proposition 23.2]{Bott_Tu}).
	
	\begin{Prop}
		As a ring, $H^*(\Gr(k,n),\mathbb{C})$ is isomorphic to the quotient algebra
		\begin{equation}
			\mathbb{C}[c_1(S),\ldots,c_k(S),c_1(Q),\ldots,c_{n-k}(Q)]/(c(S)c(Q)-1).
		\end{equation}
	\end{Prop}
	
	Denote for brevity $p_j=c_j(S)$ for $1\le j\le k$ and $q_l=c_l(Q)$ for $1\le l\le n-k$. We also let $p_0=q_0=1$, $p_j=0$ if $j>k$ or $j<0$ and $q_l=0$ if $l>n-k$ or $l<0$. Then, the proposition above states that cohomology ring of the complex Grassmannian is isomorphic to the quotient algebra $\mathcal{A}=\mathbb{C}[p_1,\ldots,p_k,q_1,\ldots,q_{n-k}]/I$. Here $I$ is the ideal in the polynomial algebra $\mathbb{C}[p_1,\ldots,p_k,q_1,\ldots,q_{n-k}]$ generated by elements $R_m$, which are defined via the polynomial identity
	\begin{equation}
		x^n+R_1x^{n-1}+\ldots+R_{n-1}x+R_{n}=(x^k+p_1x^{k-1}+\ldots+p_k)(x^{n-k}+q_1x^{n-k-1}+\ldots+q_{n-k}),
	\end{equation}
	or explicitly as
	\begin{equation*}
		R_m=\sum_{r+s=m}p_rq_s~\text{for}~1\le m\le n.
	\end{equation*}
	\begin{Rem}
		Cohomological degree on $H^*(\Gr(k,n),\mathbb{C})$ produces a grading on algebra $\mathcal{A}$. Since the Chern classes $p_j=c_j(S)$ and $q_j=c_j(Q)$ have cohomological degree $2j$ the degree of a monomial $a=\prod_{j}p_j^{\alpha_j}\prod_{l}q_{l}^{\beta_l}$ in $\mathcal{A}$ equals 
		\begin{equation}\label{degree}
			\deg(a)=2\sum_{j=1}^{k}j\alpha_j+2\sum_{l=1}^{n-k}l\beta_l.
		\end{equation}
	\end{Rem}

	\subsection{$\mathbf{GL_n(\mathbb{C})}$-structure.} 
	According to the Geometric Satake isomorphism of \cite{Mirkovic-Vilonen} there is a canonical isomorphism between the monoidal categories of irreducible representations of a complex reductive group and  of perverse sheaves on the affine Grassmanian for the Langlands dual group. A refined version \cite[7.4 Corollary]{Mirkovic-Vilonen} identifies the weight space decomposition of a maximal torus on the irreducible representation with the decomposition of the intersection cohomology of the matching affine Schubert cell in the affine Grassmannian into the classes of certain Mirkovic-Vilonen cycles.  In the case of $\mathrm{GL}_n$ and its $k$-th fundamental representation on $\bigwedge^k \mathbb{C}^n$ the corresponding affine Schubert cell is our classical Grassmanian $\Gr(k,n)$. The Mirkovic-Vilonen cycles correspond to the classical Schubert cells in $\Gr(k,n)$. \cite[7.4 Corollary]{Mirkovic-Vilonen} in this case implies that the canonical isomorphism $H^*(\Gr(k,n),\mathbb{C})\cong\bigwedge^k \mathbb{C}^n$ matches the cohomology classes of the classical Schubert cells with pure $k$-forms  in $\bigwedge^k \mathbb{C}^n$ as explained in Subsection 3.2. Namely, this isomorphism sends the Schubert class $\sigma_{\lambda}\in H^*(\Gr(k,n),\mathbb{C})$ to the basis element $\bigwedge_{j=1}^{k}e_{\lambda_j+k-j}$ of $\bigwedge^k \mathbb{C}^n$, where $\lambda$ is a partition of the length at most $k$ and whose parts are at most $n-k$ and $\{e_{j}\}_{j=0}^{n-1}$ is a basis of $\mathbb{C}^n$. We review this isomorphism in Section 3 of this paper.
	
	In particular, we can regard the cohomology ring of the complex Grassmannian as a $\mathrm{GL}_n$-representation.
	\begin{Prop}\label{cohomology_gl_n_rep}
		The cohomology ring of the complex Grassmannian can be endowed with the structure of a $\mathrm{GL}_n$-module and as a $\mathrm{GL}_{n}$-representation it is isomorphic to the $k$-th exterior power of the standard $n$-dimensional $\mathrm{GL}_n$-representation.
	\end{Prop}
	
	\begin{Rem}
		Of course, in order to prove the proposition in this formulation we only need to check that $H^*(\Gr(k,n),\mathbb{C})$ as a vector space has the dimension $\binom{n}{k}$ (the dimension of the $k$-th exterior power of an $n$-dimensional vector space). However, we will show that there are some connections between the inner structure of the cohomology ring and the $\mathrm{GL}_n$-structure defined in Subsection 3.2. Specifically, we will show that the Dynkin's grading (see Subsection 1.3 and Proposition \ref{HF_action_cohomology}) on the weight spaces corresponds to the cohomological degree.
	\end{Rem}
	
	\subsection{Dynkin's grading and principal $\mathfrak{sl}_2$-triples.} Dynkin in his paper \cite{Dynkin} considered irreducible representations of semisimple Lie groups and proved several properties of their weight systems. In that paper he also introduced what is now called `a principal $\mathfrak{sl}_2$-triple'. In this paragraph, we will give a short overview of Dynkin's results from \cite{Dynkin} which will be used in the sequel.
	
	Let $\mathfrak{g}$ be a complex semisimple Lie algebra and let $\mathfrak{h}$ be its Cartan subalgebra. Denote the set of simple roots of the Lie algebra $\mathfrak{g}$ by $S$ and the Killing form by $B(\cdot,\cdot)$.  
	
	\begin{Def}
		An \emph{$\mathfrak{sl}_2$-triple} in the Lie algebra $\mathfrak{g}$ is a triple $(e,f,h)$ of elements of $\mathfrak{g}$ which satisfy the following commutation relations:
		\begin{equation}\label{sl2_rel}
			[h,e]=2e,~[h,f]=-2f,~[e,f]=h.
		\end{equation}
		In other words, elements $e$, $f$ and $h$ span a representation of $\mathfrak{sl}_2$ in $\mathfrak{g}$.
	\end{Def}
	Elements $e$ and $f$ are often referred to as \emph{nilpotent} while operator $h$ is often referred to as \emph{semisimple}. These names are justified by the fact that in every $\mathfrak{sl}_{2}$-triple as linear operators $e$ and $f$ are nilpotent and operator $h$ is semisimple. The latter is one of consequences of the representation theory of $\mathfrak{sl}_2$. 
	
	\begin{Def}
		An $\mathfrak{sl}_2$-triple is called \emph{principal} if its elements are regular elements of the Lie algebra $\mathfrak{g}$, that is, their centralizers has the least possible dimension (see \cite[Appendix D, Def. D.2]{Fulton-Harris}). The corresponding three-dimensional Lie subalgebra spanned by the elements of the principal $\mathfrak{sl}_2$-triple is called \emph{principal sublagebra}.
	\end{Def}
	
	Let $\mathbb{V}_{\lambda}$ be the irreducible representation of $\mathfrak{g}$ with the highest weight $\lambda$ and $\rho\colon\mathfrak{g}\to\mathfrak{gl}(\mathbb{V}_{\lambda})$ be the corresponding homomorphism. Dynkin proved that every weight $\mu$ of the representation $\mathbb{V}_{\lambda}$ can be represented as
	\begin{equation}\label{weight_level}
		\mu=\lambda-\alpha_1-\ldots-\alpha_k,~\text{where}~\alpha_1,\ldots,\alpha_k\in S.  
	\end{equation}
	Moreover, for any such decomposition the number $k$ is the same. Following Dynkin, we will say that weight $\mu$ is a weight of the level $k$ if the equality \eqref{weight_level} holds for some $\alpha_1,\ldots,\alpha_k\in S$.
	
	Dynkin constructed in $\mathfrak{g}$ three special elements $e$, $f$ and $h$, where $h\in\mathfrak{h}$, which after multiplying by some constants satisfy the following relations:
	\begin{equation}\label{dynkin_cond_sl2}
		[h,e]=2e,~[h,f]=-2f,~[e,f]=h
	\end{equation}
	and also 
	\begin{equation}\label{dynkin_cond_proj}
		\alpha(h)=2~\text{for all}~\alpha\in S.
	\end{equation}
	The explicit construction of $e$, $f$ and $h$ is discussed in more detail in the Section 2 of this note. Note that triple $\{e,f,h\}$ is not defined uniquely by these conditions (see Section 2 and Remark \ref{dynkin_nonuniq}).
	
	In particular, triple $\{e,f,h\}$ spans a three-dimensional subalgebra of $\mathfrak{g}$ and form a principal $\mathfrak{sl}_2$-triple. Denote this subalgebra by $\mathfrak{s}$ and note that $\mathfrak{s}$ is isomorphic to $\mathfrak{sl}_2$. It is clear that $\mathbb{V}_{\lambda}$ can be also considered as a representation of the Lie algebra $\mathfrak{s}$.
	
	The connection between weight levels and triple $\{e,f,h\}$ is described by the following observation: the level $k$ of a weight $\mu$ in $\mathbb{V}_{\lambda}$ can be computed as
	\begin{equation}
		k=\frac{\lambda(h)-\mu(h)}{2}.
	\end{equation}
	In other words, the element $h$ produces a certain grading on the set of weights of $\mathbb{V}_{\lambda}$. Moreover, it turns out that this grading corresponds exactly to the weight decomposition of $\mathbb{V}_{\lambda}$ as an $\mathfrak{s}$-representation. Namely, consider decomposition of $\mathbb{V}_{\lambda}$ into eigenspaces of $\rho(h)$:
	\begin{equation}
		\mathbb{V}_{\lambda}=\bigoplus_{k=0}^{d}\mathbb{V}_{\lambda,d-2k},
	\end{equation}
	where $\mathbb{V}_{\lambda,r}$ is the $r$-eigenspace of $\rho(h)$, i.e.
	\begin{equation*}
		\mathbb{V}_{\lambda,r}=\ker (\rho(h)-r\cdot\id).
	\end{equation*}
	Dynkin proved that $d=\lambda(h)$ and that $\mathbb{V}_{\lambda,d-2k}$ is a direct sum of the $\mathfrak{g}$-weight spaces of $\mathbb{V}_{\lambda}$ corresponding to the weights of level $k$. 
	
	Besides that, it follows from the representation theory of $\mathfrak{sl}_2$ that eigenspaces $\mathbb{V}_{\lambda,2j-d}$ and $\mathbb{V}_{\lambda,d-2j}$ has the same dimension and moreover, the map $\rho(e)^{2(d-2j)}\colon\mathbb{V}_{\lambda,2j-d}\to\mathbb{V}_{\lambda,d-2j}$ produces a linear isomorphism between these vector spaces. Similarly, the map $\rho(f)^{2(d-2j)}\colon \mathbb{V}_{\lambda,d-2j}\to\mathbb{V}_{\lambda,2j-d}$ is also a linear isomorphism. We will observe similar effects on $H^*(\Gr(k,n),\mathbb{C})$ which are the consequences of the hard Lefschetz theorem (see Remark \ref{lefschetz}).

	\subsection{Main results.} Proposition \ref{cohomology_gl_n_rep} states that $H^*(\Gr(k,n),\mathbb{C})$ is isomorphic as a representation of $\mathrm{GL}_n$ to its $k$-th fundamental representation $\bigwedge^k\mathbb{C}^n$. 
	Recall that any representation of a Lie group produces a representation of the corresponding Lie algebra. Namely, if $G$ is a Lie algebra, $\mathfrak{g}$ is the corresponding Lie algebra and $\rho\colon G\to\mathrm{GL}(V)$ is a representation of $G$, then differential $d\rho\colon\mathfrak{g}\to\mathfrak{gl}(V)$ is a representation of $\mathfrak{g}$. In particular, Proposition \ref{cohomology_gl_n_rep} implies there is an action of the Lie algebra $\mathfrak{gl}_n(\mathbb{C})$ on the cohomology ring $H^*(\Gr(k,n),\mathbb{C})$ of the Grassmannian. Thus, we can also regard $H^*(\Gr(k,n),\mathbb{C})$ as an $\mathfrak{sl}_n(\mathbb{C})$-representation.

	Let $\{e,f,h\}$ be the Dynkin's principal $\mathfrak{sl}_2$-triple for $\mathfrak{sl}_n(\mathbb{C})$ defined as in Subsection 2.3. Denote the operators which correspond to the actions of $e$, $f$ and $h$ on cohomology by $E$, $F$ and $H$, respectively. It turns out that the actions $H$ and $F$ on the cohomology ring have clear geometrical meaning.
	
	\begin{Prop}\label{HF_action_cohomology}
		The actions of the elements $h$ and $f$ of the Dynkin's principal $\mathfrak{sl}_2$-triple of $\mathfrak{sl}_n(\mathbb{C})$ correspond to the operators $H$ and $F$ on the cohomology ring $H^*(\Gr(k,n),\mathbb{C})$ which act as follows:
		\begin{equation}\label{HF_cohomology}
			\begin{split}
				H([\omega])&=(k(n-k)-2j)\cdot [\omega],
				\\
				F([\omega])&=-c_1(S)\cdot[\omega]
			\end{split}
		\end{equation}
		for every cohomology class $[\omega]\in H^{2j}(\Gr(k,n),\mathbb{C})$, where $0\le j\le k(n-k)$. Moreover, Dynkin's grading on $H^*(\Gr(k,n),\mathbb{C})$ considered as an $\mathfrak{sl}_n(\mathbb{C})$-representation corresponds to the cohomological degree.
	\end{Prop}
	In other words, $H=k(n-k)\cdot\id-2\cdot\mathrm{Deg}$, where $\mathrm{Deg}$ is the operator of multiplication by the cohomology degree and $F$ is the operator of multiplication by the negative of the first Chern class (i.e. by $-p_1$ in terms of algebra $\mathcal{A}$).

	\begin{Rem}\label{lefschetz}
		Note that $H^2(\Gr(k,n),\mathbb{C})$ is one-dimensional, so it is generated by the first Chern class $c_1(S)=-c_1(Q)$ or by $p_1=-q_1$ in terms of $\mathcal{A}$. Thus, it follows from the hard Lefschetz theorem for the complex Grassmannian that for each $j\le\frac{k(n-k)}{2}$ the operator $F^{k(n-k)-2j}$ produces an isomorphism between $H^j(\Gr(k,n),\mathbb{C})$ and $H^{k(n-k)-j}(\Gr(k,n),\mathbb{C})$. This exactly the same effect that we have seen on weight spaces of the representations of the Lie algebra $\mathfrak{s}$ (see Subsection 1.3).
	\end{Rem}
	
	One can ask a natural question: does the action of the operator $E$ have a geometrical meaning? Unfortunately, it seems that the answer for this question is unknown. However, it is possible to describe actions of $E$, $F$ and $H$ on $H^*(\Gr(k,n),\mathbb{C})$ completely in terms of the quotient algebra $\mathcal{A}=\mathbb{C}[p_1,\ldots,p_k,q_1,\ldots,q_{n-k}]/I$ and its generators $p_i,q_j$ (recall that $\mathcal{A}$ is isomorphic to the $H^*(\Gr(k,n),\mathbb{C})$). The following proposition is the main result of this note.
	
	\begin{Prop}\label{EFH_operators}
		The actions of the operators $E$, $F$ and $H$ in terms of algebra $\mathcal{A}$ can be described via the following differential operators:
		\begin{equation}\label{EFH_diff_formulas}
			\begin{aligned}
				H = & ~ -2\sum_{j=1}^{k}jp_j\frac{\partial}{\partial p_j}-2\sum_{j=1}^{n-k}jq_j\frac{\partial}{\partial q_j} + k(n-k),
				\\
				E = & -\sum_{j=1}^{k}(k-j+1)(n-k+j-1)p_{j-1}\frac{\partial}{\partial p_{j}}+
				\\
				&+\sum_{j=1}^{n-k}(k+j-1)(n-k-j+1)q_{j-1}\frac{\partial}{\partial q_{j}}+
				\\
				&+\sum_{j,l=1}^{k}\left(\sum_{d=0}^{j-1}(j+l-2d-1)p_{d}p_{j+l-d-1}\right)\frac{\partial^2}{\partial p_j \partial p_l}-
				\\
				&-\sum_{j,l=1}^{n-k}\left(\sum_{d=0}^{j-1}(j+l-2d-1)q_dq_{j+l-d-1}\right)\frac{\partial^2}{\partial q_j \partial q_l}+
				\\
				&+2\sum_{j=1}^{k}\sum_{l=1}^{n-k}(j+l-1)\left(\sum_{d=0}^{j-1}p_{d}q_{j+l-d-1}\right)\frac{\partial^2}{\partial p_j \partial q_l},
				\\
				F = & ~ -p_1.
			\end{aligned}
		\end{equation}
		Here the elements of the algebra $\mathcal{A}$ are considered as the corresponding scalar multiplication operators, i.e. $a\in\mathcal{A}$ correspond to the operator $L_a\colon\mathcal{A}\to\mathcal{A}$, where $L_a(b)=a\cdot b$.
	\end{Prop}
	We prove propositions \ref{HF_action_cohomology} and \ref{EFH_operators} in Section 4 of this note.
	
	\begin{Rem}
		In particular, differential operators given by these formulas are  well defined on the quotient $\mathbb{C}[p_1,\ldots,p_k,q_1,\ldots,q_{n-k}]/I$, i.e. as operators on the whole polynomial algebra $\mathbb{C}[p_1,\ldots,p_k,q_1,\ldots,q_{n-k}]$ they preserve the ideal $I$. Moreover, these operators form an $\mathfrak{sl}_2$-triple.
	\end{Rem}
	
	\subsection{Overview of the proofs of propositions \ref{HF_action_cohomology} and \ref{EFH_operators}.}
	On order to prove Proposition \ref{HF_action_cohomology} we will use the explicit formulas for $e$, $f$ and $h$ given in Subsection 2.3 and the isomorphism between $\bigwedge^k\mathbb{C}^n$ and $H^*(\Gr(k,n),\mathbb{C})$. After that, it is easy to check that actions of $H$ and $F$ on cohomology correspond to the differential operators in Proposition \ref{EFH_operators}. However, the similar check for $E$ requires much more complicated computations involving many tools from the theory of symmetric functions such as Pieri, Giambelli and Jacobi-Trudi identities. This is because we do not have an analogue of formulas \eqref{HF_cohomology} for the operator $E$. Initially, only the action of $E$ on Schubert classes is known (see Proposition \ref{EFH_schubert}).

	Instead of performing these computations, we first check that differential operators in Proposition \ref{EFH_operators} are well defined on $\mathcal{A}$ and satisfy the $\mathfrak{sl}_2$ relations. The proof is straight-forward and rather computational. Since we have explicit formulas it remains to check some algebraic identities. 
	
	Next, in order to prove that the action of the differential operator corresponding to $E$ coincides with the action of $e$ on $\mathcal{A}$, we use the Jacobson-Morozov theorem.
	The latter claims that any nonzero nilpotent endomorphism of the semisimple Lie algebra $\mathfrak{g}$ can be extended to an $\mathfrak{sl}_2$-triple. Moreover, such triple is unique up to transformation from the centralizer of that nilpotent element, see \cite[Chapter 6, Proposition 2.1]{Onishchik_Vinberg}. In particular, for given $f$ and $h$ the third element $e$ of $\mathfrak{sl}_2$-triple, if exists, is defined uniquely. Applying this fact to the Lie algebra of linear operators on $\mathcal{A}$, we obtain the desired. Indeed, since $H$ and $F$ are already the images of elements $h$ and $f$ of the Dynkin's principal $\mathfrak{sl}_2$-triple and $\{E,F,H\}$ satisfy $\mathfrak{sl}_2$ commutation relations, the operator $E$ must coincide with the image of the element $e$.
	
	One drawback of this proof is that the differential expression for $E$ is appearing from nowhere. It raises the following interesting question: given a linear operator on the quotient of the polynomial algebra, how can we get the differential expression for it? We discuss the general approach for solving such problems in Section 5 of this note. 

	\subsection{Contents.} Structure of the paper is as follows. In Section 2 we review the general construction of Dynkin's principal $\mathfrak{sl}_2$-triples.
	In Section 3 we review the explicit isomorphism between the cohomology ring and the $k$-th exterior power of $\mathbb{C}^n$ which is given by means of the Schubert classes. Besides that, we also give actions of $\{e,f,h\}$ on the Schubert classes. Section 4 is devoted to the proofs of propositions \ref{EFH_operators} and \ref{HF_action_cohomology}. In Section 5 we discuss general approach for finding differential expression for a given linear operator acting on polynomial algebra (or its quotient). Finally, in Section 6 we give concluding remarks to this paper.
	
	\textbf{Acknowledgements.} This work was done during the IST Austria Summer Program and with the support of OeAD (Austria's Agency for Education and Internationalisation). The author is grateful to his supervisor Tam\'as Hausel for the constant guidance and useful advice and also to Anton Mellit for fruitful discussions.

	\section{Construction of Dynkin' principal $\mathfrak{sl}_2$-triples}
	
	In this section we discuss the construction of Dynkin' principal $\mathfrak{sl}_2$-triples. It will be useful to consider more general context.
	
	\subsection{Notation and conventions.} Let $\mathfrak{g}$ be a semisimple Lie algebra. Denote by $\ad$ the adjoint map $\ad(\cdot)(\cdot)\colon\mathfrak{g}\times\mathfrak{g}\to\mathfrak{g}$ which is defined as $\ad(X)(Y)=[X,Y]$.
	
	Let $\mathfrak{h}$ be a Cartan subalgebra of the Lie algebra $\mathfrak{g}$.
	Consider the corresponding root system and the decomposition
	\begin{equation}
		\mathfrak{g}=\mathfrak{h}\oplus\bigoplus_{\alpha\in R}\mathfrak{g}_{\alpha},
	\end{equation}
	where $R$ is the set of roots of the Lie algebra $\mathfrak{g}$. Here $\mathfrak{g}_{\alpha}$ is a root space corresponding to the root $\alpha\in\mathfrak{h}^*$, i.e. for any $H\in\mathfrak{h}$ and $X\in\mathfrak{g}_{\alpha}$ we have
	\begin{equation}
		\ad(H)(X)=\alpha(H)\cdot X,
	\end{equation}
	so the $\ad(H)$ acts diagonally on $\mathfrak{g}_{\alpha}$. In fact, each subspace $\mathfrak{g}_{\alpha}$ is one-dimensional.
	
	It is known that the set $R$ of roots is symmetric with respect to the origin and one can choose an element $l\in(\mathfrak{h}^*)^*$ such that $0\notin l(R)$ and then decompose $R$ into two subsets: $R^+$ and $R^-$ depending on the sign of $l(\alpha)$. These roots are called positive and negative, respectively. We define then the set of \emph{simple roots} as the set of those $\alpha\in R^+$ which cannot be represented as sum of two elements in $R^+$. We will denote the set of simple roots as $S$.
	
	Let $B(\cdot,\cdot)\colon\mathfrak{g}\times\mathfrak{g}\to\mathbb{C}$ be the Killing form which is defined via the formula $B(X,Y)=\tr(\ad(X)\circ\ad(Y))$. It can be shown that $B(\cdot,\cdot)$ is a nondegenerate symmetric bilinear from on the semisimple Lie algebra $\mathfrak{g}$ (and on $\mathfrak{h}$ in particular) and, hence, the Killing form produces a natural isomorphism between $\mathfrak{h}^*$ and $\mathfrak{h}$ given by the correspondence
	\begin{equation*}
		\alpha\in\mathfrak{h}^*\leftrightarrow T_{\alpha}\in\mathfrak{h},
	\end{equation*}
	where $T_{\alpha}$ is defined by
	\begin{equation}
		B(T_{\alpha},H)=\alpha(H)~\text{for all}~H\in\mathfrak{h}.
	\end{equation}
	Using this isomorphism we can define the Killing form on $\mathfrak{h}^*$ via formula 
	\begin{equation}
		B(T_{\alpha},T_{\beta})=B(\alpha,\beta).
	\end{equation}
	
	\subsection{Dynkin's construction.} Now we can construct a principal $\mathfrak{sl}_2$-triple which satisfies conditions \eqref{dynkin_cond_sl2} and \eqref{dynkin_cond_proj}.
	For each $\alpha\in S$ choose an $X_{\alpha}\in \mathfrak{g}_{\alpha}$ and $Y_{\alpha}\in\mathfrak{g}_{-\alpha}$ such that 
	\begin{equation*}
		B(X_{\alpha},Y_{\alpha})=1.
	\end{equation*}
	Next, consider elements $e$, $f$, $h$ defined by the equalities
	\begin{equation}
		h=\sum_{\alpha\in S}p_{\alpha} T_{\alpha},~e=\sum_{\alpha\in S}u_{\alpha} X_{\alpha},~f=\sum_{\alpha\in S}v_{\alpha} Y_{\alpha}
	\end{equation}
	where complex numbers $u_{\alpha}$, $v_{\alpha}$ and $p_{\alpha}$ satisfy the following conditions
	\begin{align}
		\label{cond1} & B(h,T_{\alpha})=2~\text{for all}~\alpha\in S
		\\
		\label{cond2} & u_{\alpha} v_{\alpha}=p_{\alpha}~\text{for all}~\alpha\in S.
	\end{align}
	In other words, $u_{\alpha},v_{\alpha}$ and $p_{\alpha}$ satisfy the following system
	\begin{equation}\label{dynkin_sl_system}
		\begin{cases}
			\sum\limits_{\beta\in S}B(T_{\alpha},T_{\beta})\cdot p_{\alpha}=2
			\\
			u_{\alpha}v_{\alpha}=p_{\alpha}
		\end{cases}
	\end{equation}
	for all $\alpha\in S$.
	\begin{Rem}\label{dynkin_nonuniq}
		It is clear that this principal $\mathfrak{sl}_2$-triple is not unique since the only condition on $u_{\alpha}$ and $v_{\alpha}$ is that their product equals $p_{\alpha}$. However, the element $h$ is defined uniquely by the system \eqref{dynkin_sl_system} because the matrix $\{B(T_{\alpha},T_{\beta})\}_{\alpha,\beta\in S}$ is nondegenerate. 
	\end{Rem}
	Dynkin stated in his paper \cite{Dynkin} the following fact:
	\begin{Prop}
		Operators $E,F,H$ defined via \eqref{dynkin_sl_system} satisfy conditions \eqref{dynkin_cond_sl2} and \eqref{dynkin_cond_proj}. In particular, for any finite-dimensional representation $\rho\colon\mathfrak{g}\to \mathfrak{gl}(V)$ the triple $\{\rho(E),\rho(F),\rho(H)\}$ is an $\mathfrak{sl}_2$-triple acting on $V$.
	\end{Prop}

	\subsection{Explicit formulas for the case $\mathfrak{g}=\mathfrak{sl}_n(\mathbb{C})$.}
	We will use the standard matrix representation of $\mathfrak{sl}_n$. It would be convenient for us to enumerate rows and columns by numbers from $0$ to $n-1$. For any $i,j\in\{0,1,\ldots,n-1\}$ let $E_{ij}$ be the $(i,j)$-th matrix unit.
	It can be computed that elements of Dynkin's triple in the case $\mathfrak{g}=\mathfrak{sl}_n(\mathbb{C})$ are as follows
	\begin{equation}
		\begin{aligned}
			& h=\sum_{j=0}^{n-1}(n-1-2j)E_{jj},
			\\
			& e=\sum_{j=1}^{n-1} u_j\cdot E_{j-1j},
			\\
			& f=\sum_{j=1}^{n-1}u_{j}^{-1}\cdot j(n-j)E_{jj-1},
		\end{aligned}
	\end{equation}
	where $\{u_{j}\}_{j=1}^{n-1}$ are arbitrary nonzero scalars. From now on we set $u_{j}=j(n-j)$ for all $j=\overline{1,n-1}$ in further results. In this case we have
	\begin{equation}\label{sl_2_in_sl_n}
		\begin{aligned}
			& h=\sum_{j=0}^{n-1}(n-1-2j)E_{jj},
			\\
			& e=\sum_{j=1}^{n-1} j(n-j)\cdot E_{j-1j},
			\\
			& f=\sum_{j=1}^{n-1}E_{jj-1}.
		\end{aligned}
	\end{equation}
	We will often refer to this particular triple as to the Dynkin's principal $\mathfrak{sl}_2$-triple of $\mathfrak{sl}_n(\mathbb{C})$.
	\begin{Rem}
		One can also choose $u_j=1$ instead of $u_j=j(n-j)$. This option is used for example, in \cite[Chapter 3, 3.7]{Chriss-Ginzburg}.
	\end{Rem}
	
	Now, using this formulas one can find the action of $\{e,f,h\}$ on the basis of the $k$-th exterior power $\bigwedge^k\mathbb{C}^n$. Let $V=\mathbb{C}^n$ be the standard $n$-dimensional representation of $\mathfrak{sl}_n(\mathbb{C})$ with the basis $\{e_0,\ldots,e_{n-1}\}$ and $\rho\colon\mathfrak{g}\to\mathfrak{gl}(\bigwedge^k V)$ be the corresponding representation of $\mathfrak{g}$. Then, it easy to check using \eqref{sl_2_in_sl_n} that for any basis vector $e_{j_1}\wedge\ldots\wedge e_{j_k}$ in $\bigwedge^k V$ we have
	\begin{equation}
		\begin{split}
			& \rho(h)(e_{j_1}\wedge\ldots\wedge e_{j_k})=\left(k(n-1)-2\sum_{l=1}^{k}j_l\right)\cdot e_{j_1}\wedge\ldots\wedge e_{j_k},
			\\
			& \rho(e)(e_{j_1}\wedge\ldots\wedge e_{j_k})=\sum_{l=1}^{k}j_l(n-j_l)\cdot e_{j_1}\wedge\ldots\wedge e_{j_{l-1}}\wedge e_{j_{l}+1}\wedge e_{j_{l+1}}\wedge\ldots\wedge e_{j_k},
			\\
			& \rho(f)(e_{j_1}\wedge\ldots\wedge e_{j_k})=\sum_{l=1}^{k} e_{j_1}\wedge\ldots\wedge e_{j_{l-1}}\wedge e_{j_{l}+1}\wedge e_{j_{l+1}}\wedge\ldots\wedge e_{j_k}.
		\end{split}
	\end{equation}

	\section{$\mathrm{GL}_n$-structure on the cohomology ring}
	In this section we discuss the $\mathrm{GL}_n$-structure on the cohomology ring of the complex Grassmannian. 
	\subsection{Schubert classes.}
	The cohomology ring of the Grassmannian can be also described via the Schubert calculus. We refer the reader to \cite[Chapter 14, Sections 14.6 and 14.7]{Fulton1}, \cite[Chapter 9, Section 9.4]{Fulton2} and \cite[Sections 3 and 4]{Gillespie} for more details.
	
	Let us first recall some notions from the theory of symmetric functions. A \textbf{partition} is a nonincreasing sequence of nonnegative integers $\lambda=(\lambda_1,\ldots,\lambda_m)$. 
	The \textbf{size} of the partition $\lambda$, denoted $|\lambda|$, is $\sum_{j}\lambda_j$, the entries of $\lambda$ are called its \textbf{parts} and the number of nonzero parts, denoted $\ell(\lambda)$, is called the \textbf{length} of $\lambda$. The \textbf{Young diagram} of a partition $\lambda$ is the left-aligned partial grid of boxes in which the $i$-th row from the top has $\lambda_i$ boxes.
	Let $\mathcal{P}_{k,n}$ be the set of partitions $\lambda$ of the length at most $k$ and whose parts are at most $n-k$.
	
	\begin{Rem}
		In terms of the Young diagrams the set $\mathcal{P}_{k,n}$ consists of all partitions $\lambda$, whose Young diagram can be embedded into the rectangle $k\times(n-k)$.
	\end{Rem}

	Fix an arbitrary complete flag $\mathcal{F}=(F_{j})_{j=1}^{n}$ in the initial $n$-dimensional vector space $W$:
	\begin{equation}
		0\subset F_1\subset\ldots\subset F_{n-1}\subset F_{n}=W.
	\end{equation}
	For each partition $\lambda\in\mathcal{P}_{k,n}$ the \textbf{Schubert subvariety} $\Omega_{\lambda}$ associated to $\mathcal{F}$
	is defined as
	\begin{equation}
		\Omega_{\lambda}=\{U\in\Gr(k,n)\colon\dim (U\cap F_r)=j~\text{for}~n-k+j-\lambda_j\le n-k+j-\lambda_{j+1}\}.
	\end{equation}
	The corresponding \textbf{Schubert class} $\sigma_{\lambda}$ is defined as the cohomology class dual to $\Omega_{\lambda}$.
	
	The following proposition describes multiplication in the cohomology ring in terms of Schubert classes.
	\begin{Prop}\label{symm_func_isom}
		Schubert classes $\{\sigma_{\lambda}\}_{\lambda\in\mathcal{P}_{k,n}}$ form $\mathbb{C}$-basis in the cohomology ring $H^*(\Gr(k,n),\mathbb{C})$. The ring structure is given by the Littlewood-Richardson coefficients $c_{\lambda\mu}^{\nu}$:
		\begin{equation}
			\sigma_{\lambda}\cdot\sigma_{\mu}=\sum_{\nu}c_{\lambda\mu}^{\nu}\cdot\sigma_{\nu}.
		\end{equation}
		Moreover, there is an isomorphism 
		\begin{equation}
			H^*(\Gr(k,n),\mathbb{C})\simeq\Lambda[x_1,x_2,\ldots]/(s_{\lambda}\mid\lambda\notin\mathcal{P}_{k,n}),
		\end{equation}
		where $\Lambda[x_1,x_2,\ldots]$ is the ring of the symmetric functions in infinitely many variables; under this isomorphism the Schubert class $\sigma_{\lambda}$ maps to the Schur function $s_{\lambda}$ which corresponds to the partition $\lambda$.
	\end{Prop}
	
	In this description of the cohomology ring of the Grassmannian the Chern classes correspond to the elementary and complete symmetric functions. Namely, under the isomorphism mentioned in Proposition \ref{symm_func_isom} we have the following correspondence:
	\begin{equation}\label{chern_schubert}
		\begin{split}
			& p_{j}\longleftrightarrow(-1)^je_{j}=(-1)^js_{(1,\ldots,1)},
			\\
			& q_{j}\longleftrightarrow h_j=s_{(j)},
		\end{split}
	\end{equation}
	where $e_j$ is the $j$-th elementary symmetric polynomial and $h_j$ is the $j$-th complete symmetric polynomial (see \cite{Macdonald} for more details).
	
	\subsection{Isomorphism with the $k$-th exterior power of $\mathbb{C}^n$.}
	Let $V=\mathbb{C}^n$ be an $n$-dimensional vector space with basis  $\{e_0,\ldots,e_{n-1}\}$. Now, following \cite[Section 1.3]{Gatto}, let us assign to each partition $\lambda=(\lambda_1,\ldots,\lambda_k)$, where $\lambda\in\mathcal{P}_{k,n}$, the basis element $\bigwedge\limits_{j=1}^{k}e_{\lambda_j+k-j}$ of the $k$-th exterior power $\bigwedge^k V$. Note that while $\lambda$ runs over the set $\mathcal{P}_{k,n}$, the sequence $(\lambda_j+k-j)_{j=1}^{k}$ runs over all strictly decreasing sequences with elements from the set $\{0,1,\ldots,n-1\}$. 
	
	Thus, we defined a linear isomorphism between $H^*(\Gr(k,n),\mathbb{C})$ and $\bigwedge^k\mathbb{C}^n$, so now we can endow the cohomology ring of $\Gr(k,n)$ with a structure of a $\mathrm{GL}_n$-representation via this correspondence. 
	
	As was mentioned in Subsection 2.3, the elements of the Dynkin's $\mathfrak{sl}_2$-triple of $\mathfrak{sl}_n(\mathbb{C})$ act on basis vectors of $\bigwedge^k V$ in the following way:
	\begin{equation}\label{sl_2_exterior}
		\begin{split}
			& \rho(h)(e_{j_1}\wedge\ldots\wedge e_{j_k})=(k(n-1)-2(j_1+\ldots+j_k))\cdot e_{j_1}\wedge\ldots\wedge e_{j_k},
			\\
			& \rho(e)(e_{j_1}\wedge\ldots\wedge e_{j_k})=\sum_{l=1}^{k}j_l(n-j_l)\cdot e_{j_1}\wedge\ldots\wedge e_{j_{l-1}}\wedge e_{j_l-1}\wedge e_{j_{l+1}}\wedge\ldots\wedge e_{j_k},
			\\
			& \rho(f)(e_{j_1}\wedge\ldots\wedge e_{j_k})=\sum_{l=1}^{k} e_{j_1}\wedge\ldots\wedge e_{j_{l-1}}\wedge e_{j_l+1}\wedge e_{j_{l+1}}\wedge\ldots\wedge e_{j_k}.
		\end{split}
	\end{equation}
	Here we define $e_{-1}=e_n=0$.
	
	Now using the isomorphism between $\bigwedge^k V$ and $H^*(\Gr(k,n),\mathbb{C})$ we can describe the action of the triple $\{E,F,H\}$ on Schubert classes. The following proposition is an immediate consequence of formulas \eqref{sl_2_exterior} and the aforementioned isomorphism.
	\begin{Prop}\label{EFH_schubert}
		Operators $E$, $F$ and $H$ act on Schubert classes $\sigma_{\lambda}$ in the following way: for every $\lambda\in\mathcal{P}_{k,n}$
		\begin{equation}\label{sl_2_schubert_classes}
			\begin{split}
				& H(\sigma_{\lambda})=(-2|\lambda|+k(n-k))\cdot\sigma_{\lambda},
				\\
				& E(\sigma_{\lambda})=\sum_{l}^{}(k+\lambda_l-l)(n-k-\lambda_l+l)\cdot\sigma_{\lambda^{(l)}},
				\\
				& F(\sigma_{\lambda})=\sum_{\mu}\sigma_{\mu}.
			\end{split}
		\end{equation}
		Here $\mu$ runs over all partitions from $\mathcal{P}_{k,n}$ which can be obtained from $\lambda$ by adding one box and $l$ runs over all indices from $\{1,2,\ldots,k\}$ for which the sequence
		\begin{equation*}
			\lambda^{(l)}=(\lambda_1,\ldots,\ldots,\lambda_{l-1},\lambda_{l}-1,\lambda_{l+1},\ldots,\lambda_k)
		\end{equation*}
		is a partition from $\mathcal{P}_{k,n}$.
	\end{Prop}

	\begin{Rem}
		We can use this isomorphism in different way. Observe that we can endow the $k$-th exterior power of the standard representation $V$ of $\mathrm{GL}_n$ (or $\mathrm{SL}_n$) with a certain ring structure. 
		It turns out that there is a class of irreducible representations of semisimple algebraic groups called \textbf{weight multiplicity free} representations that also possess a similar ring structure. Moreover, weight multiplicity free representations often arise as cohomology rings of certain projective varieties. For instance, for a minuscule representation the corresponding projective variety is a generalised flag variety. We refer the reader to the paper \cite{Panyushev} for more details. It can be verified that our ring strucure on $\bigwedge^k V$ coincides with the one defined in \cite[Section 5 and Theorem 5.5]{Panyushev}.
	\end{Rem}

	\section{Proofs of Propositions \ref{HF_action_cohomology} and \ref{EFH_operators}}

	In this section we prove the propositions \ref{HF_action_cohomology} and \ref{EFH_operators}. We will need the following simple computational lemma for both propositions (here $[\cdot,\cdot]$ denotes the commutator and $\delta_{km}$ is the Kronecker delta):
	
	\begin{Lm}\label{commut_comp}
		a) For all possible indices $j$, $l$, $r$ and $s$ the following commutation relations hold
		\begin{equation*}
			\begin{aligned}
				&\left[\frac{\partial}{\partial p_j},  p_rq_s\right]=\delta_{jr}q_s,~
				\left[\frac{\partial^2}{\partial p_j\partial p_l}, p_rq_s\right]=\delta_{lr}q_s\frac{\partial}{\partial p_j}+\delta_{jr}q_s\frac{\partial}{\partial p_l},
				\\
				&\left[\frac{\partial}{\partial q_j},  p_rq_s\right]=\delta_{js}p_r,~
				\left[\frac{\partial^2}{\partial q_j\partial q_l}, p_rq_s\right]=\delta_{ls}p_r\frac{\partial}{\partial q_j}+\delta_{js}p_r\frac{\partial}{\partial q_l},
				\\
				& \left[\frac{\partial^2}{\partial p_j\partial q_l}, p_rq_s\right]=\delta_{ls}p_r\frac{\partial}{\partial p_j}+\delta_{jr}q_s\frac{\partial}{\partial q_l}+\delta_{jr}\delta_{ls}.
			\end{aligned}
		\end{equation*}
		b) For any nonnegative integers $\alpha$ and $\beta$
		\begin{equation*}
			\left[x\frac{\partial}{\partial x},x^{\alpha}\frac{\partial^{\beta}}{\partial x^{\beta}}\right]=(\alpha-\beta)\cdot x^{\alpha}\frac{\partial^{\beta}}{\partial x^{\beta}}~\text{and}~\left[\frac{\partial^{\alpha}}{\partial x^{\alpha}},x\right]=\alpha\cdot\frac{\partial^{\alpha-1}}{\partial x^{\alpha-1}}.
		\end{equation*}
	\end{Lm}
	
	\begin{proof}[Proof of Proposition \ref{HF_action_cohomology}.] Proposition \ref{EFH_schubert} gives the explicit formulas for the actions of the elements $h$ and $f$ on the Schubert classes. Now we will simply check that the action of operators from Proposition \ref{HF_action_cohomology} on Schubert classes gives the same result as in Proposition \ref{EFH_schubert}. Denote temporarily the operators from Proposition \ref{HF_action_cohomology} as $H'$ and $F'$. We need to prove that $H=H'$ and $F=F'$. Since $\{\sigma_{\lambda}\}_{\lambda\in\mathcal{P}_{k,n}}$ form a basis of $H^*(\Gr(k,n),\mathbb{C})$ (see Proposition \ref{symm_func_isom}) it is enough to check that the actions of $H, H'$ and $F,F'$ on $\sigma_{\lambda}$ are the same.
		
		For $H$ it is sufficient to notice that according to Proposition \ref{EFH_schubert} we have
		\begin{equation}
			H(\sigma_{\lambda})=(-2|\lambda|+k(n-k))\cdot\sigma_{\lambda}.
		\end{equation}
		Note that the Schubert class $\sigma_{\lambda}$ has cohomological degree $2|\lambda|$ and hence we have $H'(\sigma_{\lambda})=H(\sigma_{\lambda})$ (see \ref{HF_action_cohomology}). 
		
		Similarly, the action of $F'$ is the multiplication by $-p_1$ which is the Schubert class of the partition $(1,0,\ldots,0)$. Hence,
		\begin{equation}
			F'(\sigma_{\lambda})=\sigma_{(1,0,\ldots,0)}\cdot\sigma_{\lambda}.
		\end{equation}
		Now note that for the corresponding Schur functions we have the following identity which is a particular case of the Pieri's rule  (see \cite[Chapter I, 5.16]{Macdonald}):
		\begin{equation}
			s_{(1,0,\ldots,0)}\cdot s_{\lambda}=\sum_{\mu}s_{\mu},
		\end{equation}
		where $\mu$ runs over all partitions which can be obtained by adding one box to $\lambda$. In view of Proposition \ref{symm_func_isom} it means that
		\begin{equation}
			\sigma_{(1,0,\ldots,0)}\cdot\sigma_{\lambda}=\sum_{\mu}\sigma_{\mu},
		\end{equation}
		where $\mu$ runs over all partitions from $\mathcal{P}_{k,n}$ which can be obtained by adding one box to $\lambda$. Thus, 
		\begin{equation}
			F'(\sigma_{\lambda})=\sum_{\mu}\sigma_{\mu}=F(\sigma_{\lambda})
		\end{equation}
		for all $\lambda\in\mathcal{P}_{k,n}$ and consequently $F=F'$.
		
		Finally, to check that Dynkin's grading on $H^*(\Gr(k,n),\mathbb{C})$ corresponds to the cohomological degree it suffices to observe that $H^{2j}(\Gr(k,n),\mathbb{C})$ is the $(k(n-k)-2j)$-eigenspace for the operator $H$. Therefore, the decomposition $H^{*}(\Gr(k,n),\mathbb{C})=\bigoplus_{j=0}^{k(n-k)}H^{2j}(\Gr(k,n),\mathbb{C})$ is the weight decomposition of $H^*(\Gr(k,n),\mathbb{C})$ as an $\mathfrak{s}$-representation.
	\end{proof}
	
	\begin{Rem}\label{jl_symmetry}
		At first glance, it might seem that coefficients of $\frac{\partial^2}{\partial p_j\partial p_l}$ and $\frac{\partial^2}{\partial q_j\partial q_l}$ in the expression of $E$ (see \eqref{EFH_diff_formulas}) are not symmetric with respect to $j$ and $l$. However, these coefficients $\sum\limits_{d=0}^{j-1}(j+l-2d-1)p_{d}p_{j+l-d-1}$ and $\sum\limits_{d=0}^{j-1}(j+l-2d-1)q_{d}q_{j+l-d-1}$ are indeed symmetric with respect to $j$ and $l$ since
		\begin{equation}
			\begin{split}
				\sum_{d=0}^{j-1}(j+l-2d-1)p_{d}p_{j+l-d-1}
				&=\sum_{d=0}^{j-1}(j+l-d-1)p_{d}p_{j+l-d-1}
				-\sum_{d=0}^{j-1}dp_{d}p_{j+l-d-1}=
				\\
				&=\sum_{d=0}^{l}dp_{d}p_{j+l-d-1}-\sum_{d=0}^{j-1}dp_{d}p_{j+l-d-1}=
				\\
				&=\sum_{d=0}^{j}dp_{d}p_{j+l-d-1}-\sum_{d=0}^{l}dp_{d}p_{j+l-d-1}=
				\\
				&=\sum_{d=0}^{l-1}(j+l-2d-1)p_{d}p_{j+l-d-1}.
			\end{split}
		\end{equation}
		In addition, we can make the expression $\sum\limits_{d=0}^{j-1}p_{d}q_{j+l-d-1}$, which is the coefficient of $\frac{\partial^2}{\partial p_j\partial q_l}$ in $E$, symmetric with respect to the variables $p_j$ and $q_l$ in $\mathcal{A}$ since we have
		\begin{equation}
			\begin{split}
				\sum_{d=0}^{j-1}p_{d}q_{j+l-d-1}
				&=\sum_{d=0}^{j+l-1}p_{d}q_{j+l-d-1}-\sum_{d=j}^{j+l-1}p_{d}q_{j+l-d-1}=
				\\
				& =R_{j+l-1}-\sum_{d=0}^{l-1}p_{d}q_{j+l-d-1}\equiv-\sum_{d=0}^{l-1}p_{d}q_{j+l-d-1}\pmod I.
			\end{split}
		\end{equation}
		Thus, the coefficient of $\frac{\partial^2}{\partial p_j\partial q_l}$ in $E$ can be rewritten as
		\begin{equation*}
			\sum_{d=0}^{j-1}p_{d}q_{j+l-d-1}\equiv\frac{1}{2}\left(\sum_{d=0}^{j-1}p_{d}q_{j+l-d-1}-\sum_{d=0}^{l-1}p_{d}q_{j+l-d-1}\right)\pmod I.
		\end{equation*}
		We use this observations in the proof of Proposition \ref{EFH_operators}.
	\end{Rem}
	
	\begin{proof}[Proof of Proposition \ref{EFH_operators}]
		Denote differential operators defined in Proposition \ref{EFH_operators} which correspond to $E$, $F$ and $H$ as $D_e$, $D_f$ and $D_h$, respectively.
		Firstly, note that $D_e$, $D_f$ and $D_h$ are linear operators on the polynomial algebra $\mathbb{C}[p_1,\ldots,p_k,q_1,\ldots,q_{n-k}]$. 
		We will say that a differential operator $D$ on $\mathbb{C}[p_1,\ldots,p_k,q_1,\ldots,q_{n-k}]$ is zero modulo $I$ if all its coefficients are zero modulo $I$. In other words, $D$ is zero modulo $I$ if it can be represented as linear combination
		\begin{equation}
			\sum_{\alpha,\beta}a_{\alpha\beta}\cdot\prod_{j=1}^{k}\frac{\partial^{\alpha_j}}{\partial p_{j}^{\alpha_j}}\prod_{l=1}^{n-k}\frac{\partial^{\beta_l}}{\partial q_{l}^{\beta_j}},
		\end{equation}
		where $a_{\alpha\beta}\in I$ for all $\alpha=(\alpha_1,\ldots,\alpha_k)$ and $\beta=(\beta_1,\ldots,\beta_{n-k})$. In particular, if $D$ is zero modulo $I$, then its image is contained in $I$.
		The proof of Proposition \ref{EFH_operators} is divided into three steps.

		\textbf{Step 1.} We prove that differential operators $D_e$, $D_f$ and $D_h$ are well defined operators on the quotient $\mathcal{A}=\mathbb{C}[p_1,\ldots,p_k,q_1,\ldots,q_{n-k}]/I$.
		
		In order to prove that $D_e,D_f,D_h$ are well defined on the quotient $\mathbb{C}[p_1,\ldots,p_k,q_1,\ldots,q_{n-k}]/I$ we need to check that they preserve the ideal $I$. Clearly, for this it is sufficient to check that for all $m$ the commutators $[D_e,R_m]$, $[D_f,R_m]$ and $[D_h,R_m]$ are zero modulo $I$ (here we consider $R_m$ as a scalar multiplication operator). Indeed, in this case for any elements $f_m\in\mathbb{C}[p_1,\ldots,p_k,q_1,\ldots,q_{n-k}]$ we have
		\begin{equation}
			D_h\left(\sum_{m}f_m R_m\right)=\sum_{m}D_h(R_mf_m)=\sum_{m}\big([D_h,R_m](f_m)+R_m\cdot D_h(f_m)\big)\equiv\sum_{m}R_m\cdot D_h(f_{m})\equiv 0\pmod I 
		\end{equation}
		which implies that $D_h(I)\subset I$ (similarly for $D_e$ and $D_f$). 
		
		Now we proceed to the computation of these commutators. Since $D_f$ is a scalar operator, the equality $[D_f,R_m]=0$ is trivial. It remains to compute $[D_h,R_m]$ and $[D_f,R_m]$ (here we use Lemma \ref{commut_comp}).
		\begin{equation}
			\begin{split}
				[D_h,R_m]=
				&-2\sum_{j=1}^{k}\left[jp_j\frac{\partial}{\partial p_j},R_m\right]-2\sum_{j=1}^{n-k}\left[jq_j\frac{\partial}{\partial q_j},R_m\right]=-2\sum_{j=1}^{k}\left[jp_j\frac{\partial}{\partial p_j},p_jq_{m-j}\right]-
				\\
				&-2\sum_{j=1}^{n-k}\left[jq_j\frac{\partial}{\partial q_j},p_{m-j}q_{j}\right]=-2\sum_{j=1}^{k}jp_jq_{m-j}-2\sum_{j=1}^{n-k}jp_{m-j}q_{j}=
				\\
				&=-2\sum_{j=0}^{m}jp_jq_{m-j}-2\sum_{j=0}^{m}jp_{m-j}q_{j}=-2m\sum_{j=0}^{m}p_{j}q_{m-j}=-2m\cdot R_m\equiv 0\pmod I.
			\end{split}
		\end{equation}
		Similar computation for $[D_e,R_m]$ is more complicated. 
		Since $[D_f,R_m]=\sum\limits_{r+s=m}[D_f,p_rq_s]$ it is useful to compute $[D_f,p_rq_s]$ first. Lemma \ref{commut_comp} implies that
		\begin{equation}
			\begin{split}
				[D_e,p_rq_s]=
				&-\sum_{j=1}^{k}(k-j+1)(n-k+j-1)\delta_{jr}p_{j-1}q_s
				+\sum_{j=1}^{n-k}(k+j-1)(n-k-j+1)\delta_{js}p_rq_{j-1}+
				\\
				&+\sum_{j,l=1}^{k}\left(\sum_{d=0}^{j-1}(j+l-2d-1)p_{d}p_{j+l-d-1}\right)\left(\delta_{lr}q_s\frac{\partial}{\partial p_j}-\delta_{jr}q_s\frac{\partial}{\partial p_l}\right)-
				\\
				&-\sum_{j,l=1}^{n-k}\left(\sum_{d=0}^{j-1}(j+l-2d-1)q_dq_{j+l-d-1}\right)\left(\delta_{ls}p_r\frac{\partial}{\partial q_j}-\delta_{js}p_r\frac{\partial}{\partial q_l}\right)+
				\\
				&+2\sum_{j=1}^{k}\sum_{l=1}^{n-k}(j+l-1)\left(\sum_{d=0}^{j-1}p_{d}q_{j+l-d-1}\right)\left(\delta_{ls}p_r\frac{\partial}{\partial p_j}+\delta_{jr}q_s\frac{\partial}{\partial q_l}+\delta_{jr}\delta_{ls}\right).
			\end{split}
		\end{equation}
		Thus, modulo $I$ we have (here we use the symmetry mentioned in Remark \ref{jl_symmetry})
		\begin{equation*}
			\begin{split}
				[D_e,p_rq_s]=
				&-(k-r+1)(n-k+r-1)p_{r-1}q_s+(k+s-1)(n-k-s+1)p_{r}q_{s-1}+
				\\
				&+2\sum_{j=1}^{k}\sum_{d=0}^{j-1}(j+r-2d-1)p_{d}p_{j+r-d-1}q_s\frac{\partial}{\partial p_j}
				-2\sum_{l=1}^{n-k}\sum_{d=0}^{l-1}(l+s-2d-1)q_dq_{l+s-d-1}p_r\frac{\partial}{\partial q_l}+
				\\
				&+2\sum_{j=1}^{k}(j+s-1)\left(\sum_{d=0}^{j-1}p_{d}q_{j+s-d-1}\right)p_r\frac{\partial}{\partial p_j}
				-2\sum_{l=1}^{n-k}(r+l-1)\left(\sum_{d=0}^{l-1}p_{r+l-d-1}q_{d}\right)q_s\frac{\partial}{\partial q_l}+
				\\
				&+2(r+s-1)\sum_{d=0}^{r-1}p_{d}q_{r+s-d-1}.
			\end{split}
		\end{equation*}
		Here we used the fact that $\sum\limits_{d=0}^{r-1}p_{d}q_{r+l-d-1}=-\sum\limits_{d=0}^{l-1}p_{r+l-d-1}q_{d}$ modulo $I$.
		Rearranging the previous sum, we see that
		\begin{equation}
			[D_e,p_rq_s]=2\sum_{j=1}^{k}A^{(rs)}_j\frac{\partial}{\partial p_j}-2\sum_{l=1}^{n-k}B^{(rs)}_l\frac{\partial}{\partial q_l}+C^{(rs)},    
		\end{equation}
		where $A^{(rs)}_{j}$, $B^{(rs)}_{j}$ and $C^{(rs)}_{j}$ are defined as
		\begin{equation}
			\begin{aligned}
				A^{(rs)}_{j}&=\sum_{d=0}^{j-1}(j+r-2d-1)p_{d}p_{j+r-d-1}q_s+(j+s-1)\sum_{d=0}^{j-1}p_{d}q_{j+s-d-1}p_r, 
				\\
				B^{(rs)}_{j}&=\sum_{d=0}^{l-1}(l+s-2d-1)q_dq_{l+s-d-1}p_r+(r+l-1)\sum_{d=0}^{l-1}p_{r+l-d-1}q_{d}q_s, 
				\\
				C^{(rs)}&=-(k-r+1)(n-k+r-1)p_{r-1}q_s+(k+s-1)(n-k-s+1)p_{r}q_{s-1}+
				\\
				& +2(r+s-1)\sum_{d=0}^{r-1}p_{d}q_{r+s-d-1}.
			\end{aligned}    
		\end{equation}
		
		Now, in order to prove that $[D_e,R_m]\equiv 0\pmod I$, it suffices to check that
		\begin{equation}
			\sum_{r+s=m}A^{(rs)}_{j}\equiv 0,~\sum_{r+s=m}B^{(rs)}_{j}\equiv 0,~\sum_{r+s=m}C^{(rs)}\equiv 0 \pmod I.
		\end{equation}
		We compute the sum of $C^{(rs)}$ first. Note that
		\begin{equation}
			\begin{split}
				\sum_{r+s=m}C^{(rs)}=
				&-\sum_{r+s=m}(k-r+1)(n-k+r-1)p_{r-1}q_s+
				\\
				&+\sum_{r+s=m}(k+s-1)(n-k-s+1)p_{r}q_{s-1}+2(m-1)\sum_{d=0}^{r-1}p_{d}q_{m-d-1}.
			\end{split}
		\end{equation}
		Since
		\begin{equation}
			\begin{split}
				&-\sum_{r+s=m}(k-r+1)(n-k+r-1)p_{r-1}q_s+\sum_{r+s=m}(k+s-1)(n-k-s+1)p_{r}q_{s-1}=
				\\
				&=-\sum_{u+v=m-1}(k-u)(n-k+u)p_{u}q_{v}+\sum_{u+v=m-1}(k+v)(n-k-v)p_{u}q_{v}=
				\\
				&=\sum_{u+v=m-1}(n-2k+u-v)(u+v)p_uq_v=(m-1)\sum_{u+v=m-1}(n-2k+u-v)p_uq_v=
				\\
				&=(m-1)(n-2k)R_{m-1}+(m-1)\sum_{u+v=m-1}(u-v)p_uq_v\equiv(m-1)\sum_{u+v=m-1}(u-v)p_uq_v\pmod I
			\end{split}
		\end{equation}
		and
		\begin{equation}
			\sum_{r+s=m}\sum_{d=0}^{r-1}p_{d}q_{m-d-1}=\sum_{d=0}^{m-1}(m-d)p_dq_{m-d-1}=\sum_{u+v=m-1}(v+1)p_uq_v,
		\end{equation}
		we have
		\begin{equation}
			\begin{split}
				\sum_{r+s=m}C^{(rs)}
				&\equiv(m-1)\sum_{u+v=m-1}(u-v)p_uq_v+2(m-1)\sum_{u+v=m-1}(v+1)p_uq_v=
				\\
				&=(m-1)\sum_{u+v=m-1}(u+v+2)p_uq_v=(m^2-1)R_{m-1}\equiv 0\pmod I
			\end{split}
		\end{equation}
		for all $m\ge 1$.
		
		Next we check that $\sum\limits_{r+s=m}A^{(rs)}_{j}\equiv 0\pmod I$ for all $m\ge 1$ (the proof for $B^{(rs)}_{j}$ is similar). 
		Indeed, we have
		\begin{equation}\label{A_main}
			\sum_{r+s=m}A^{(rs)}_{j}=\sum_{r+s=m}\sum_{d=0}^{j-1}(j+r-2d-1)p_{d}p_{j+r-d-1}q_s+\sum_{r+s=m}(j+s-1)\sum_{d=0}^{j-1}p_{d}q_{j+s-d-1}p_r.
		\end{equation}
		Now note that modulo $I$ we have (here we make substitutions in summations: $u=s-l-1$, $v=m-s$, $w=l$)
		\begin{equation}\label{A_aux_1}
			\begin{split}
				&\sum_{r+s=m}(j+s-1)\sum_{d=0}^{j-1}p_{d}q_{j+s-d-1}p_r\equiv
				-\sum_{r+s=m}(j+s-1)\sum_{l=0}^{s-1}p_{j+s-l-1}q_{l}p_r=
				\\
				&=[u=s-l-1,~v=m-s,~w=l,~u+v+w=m-1,~u,v,w\ge 0]=
				\\
				&=-\sum_{u+v+w=m-1}(j+u+w)p_{j+u}p_{v}q_{w} \pmod I.
			\end{split}
		\end{equation}
		On the other hand, 
		\begin{equation}\label{A_aux_2}
			\begin{split}
				\sum_{r+s=m}\sum_{d=0}^{j-1}(j+r-2d-1)p_{d}p_{j+r-d-1}q_s
				&=\sum_{r+s=m}q_{s}\left(\sum_{d=0}^{j-1}(j+r-d-1)p_{d}p_{j+r-d-1}-\sum_{d=0}^{j-1}dp_dp_{j+r-d-1}\right)=
				\\
				&=\sum_{r+s=m}q_{s}\left(\sum_{d=r}^{j+r-1}dp_{d}p_{j+r-d-1}-\sum_{d=0}^{j-1}dp_dp_{j+r-d-1}\right)=
				\\
				&=\sum_{r+s=m}q_{s}\left(\sum_{d=j}^{j+r-1}dp_{d}p_{j+r-d-1}-\sum_{d=0}^{r-1}dp_dp_{j+r-d-1}\right)=
				\\
				&=\sum_{r+s=m}\sum_{d=0}^{r-1}(j+d)p_{j+d}q_{r-d-1}q_{s}-\sum_{r+s=m}\sum_{d=0}^{r-1}dp_{d}p_{j+r-d-1}q_{s}.
			\end{split}
		\end{equation}
		Now we can rearrange the last expression in a way similar to \eqref{A_aux_1}. Note that
		\begin{equation}\label{A_aux_3}
			\begin{split}
				\sum_{r+s=m}\sum_{d=0}^{r-1}(j+d)p_{j+d}q_{r-d-1}q_{s}
				&=[u=d,~v=s,~w=r-d-1,~u,v,w\ge 0]=
				\\
				&=[u+v+w=m-1]=\sum_{u+v+w=m-1}(j+u)p_{j+u}q_{v}p_{w}
			\end{split}
		\end{equation}
		and
		\begin{equation}\label{A_aux_4}
			\begin{split}
				\sum_{r+s=m}\sum_{d=0}^{r-1}dp_{d}p_{j+r-d-1}q_{s}
				&=[u=r-d-1,~v=d,~w=s,~u,v,w\ge 0]=
				\\
				&=[u+v+w=m-1]=\sum_{u+v+w=m-1}vp_{j+u}q_{v}p_{w}.
			\end{split}
		\end{equation}
		Finally, combining equalities \eqref{A_main} and $\eqref{A_aux_1}-\eqref{A_aux_4}$ we obtain
		\begin{equation}
			\begin{split}
				\sum_{r+s=m}A^{(rs)}_{j}\equiv
				&\sum_{u+v+w=m-1}p_{j+u}q_{v}p_{w}(j+u-v-j-u-w)=
				\\
				&=\sum_{u+v+w=m-1}p_{j+u}q_{v}p_{w}(-v-w)=
				\\
				&=\sum_{u=0}^{m-1}p_{j+u}(u-m+1)\sum_{v+w=m-1-u}p_{w}q_{v}=
				\\
				&=\sum_{u=0}^{m-1}(u-m+1)p_{j+u}\cdot R_{m-1-u}\equiv 0\pmod I,
			\end{split}
		\end{equation}
		as desired. The congruence $\sum\limits_{r+s=m}B^{(rs)}_{j}\equiv 0\pmod I$ can be checked in the same way.
		
		Thus, we proved that differential operators $[D_h,R_m]$, $[D_e,R_m]$ and $[D_f,R_m]$ are zero modulo $I$. Hence, $D_e$, $D_f$ and $D_h$ are well defined on the quotient $\mathcal{A}=\mathbb{C}[p_1,\ldots,p_k,q_1,\ldots,q_{n-k}]/I$.
		
		\textbf{Step 2.} We check that triple $\{D_e,D_f,D_h\}$ satisfy the $\mathfrak{sl}_2$-commutation relations \eqref{sl2_rel}. In other words, we need to check the following eqialities:
		\begin{equation}
			[D_e,D_f]=D_h,~[D_h,D_e]=2D_e,~[D_h,D_f]=-2D_f.
		\end{equation} 
		Commutators $[D_e,D_f]$ and $[D_h,D_f]$ can be easily computed directly. This is because $D_f$ is the scalar multiplication by $-p_1$, so we only need to take terms containing differentiation with respect to $p_1$ into account. Indeed, Lemma \ref{commut_comp} implies that
		\begin{equation*}
			[D_h,D_f]=\left[-2p_1\frac{\partial}{\partial p_1}, -p_1\right]=2p_1=2D_f~\text{and}
		\end{equation*}
		\begin{equation}
			\begin{split}
				[D_e,D_f]=
				&-[D_e,p_1]=\sum_{j=1}^{k}(k-j+1)(n-k+j-1)p_{j-1}\left[\frac{\partial}{\partial p_{j}},p_1\right]-
				\\
				&-\sum_{j,l=1}^{k}\left(\sum_{d=0}^{j-1}(j+l-2d-1)p_{d}p_{j+l-d-1}\right)\left[\frac{\partial^2}{\partial p_j \partial p_l},p_1\right]-
				\\
				&-2\sum_{j=1}^{k}\sum_{l=1}^{n-k}(j+l-1)\left(\sum_{d=0}^{j-1}p_{d}q_{j+l-d-1}\right)\left[\frac{\partial^2}{\partial p_j \partial q_l},p_1\right]=
				\\
				&=k(n-k)-2\sum_{l=1}^{k}lp_{l}\frac{\partial}{\partial p_{l}}-2\sum_{l=1}^{n-k}lq_{l}\frac{\partial}{\partial q_{l}}=D_h.
			\end{split}
		\end{equation}
		In order to prove that $[D_h,D_e]=2D_e$ let us make the following observation first. Lemma \ref{commut_comp} implies that for any $2n$ nonnegative integers $\alpha_1$, \ldots, $\alpha_k$, $\beta_1$, \ldots, $\beta_k$ and $\gamma_1$, \ldots, $\gamma_{n-k}$, $\delta_1$ \ldots, $\delta_{n-k}$ for the operator
		\begin{equation*}
			D=\left(\prod_{j=1}^{k}p_{j}^{\alpha_{j}}\prod_{l=1}^{n-k}q_{l}^{\gamma_{l}}\right)\cdot\prod_{j=1}^{k}\left(\frac{\partial}{\partial p_{j}}\right)^{\beta_{j}}\prod_{l=1}^{n-k}\left(\frac{\partial}{\partial q_{l}}\right)^{\delta_{l}}
		\end{equation*}
		we have
		\begin{equation}
			[D_h,D]=\left(\sum_{j=1}^{k}2j(\beta_{j}-\alpha_{j})+\sum_{l=1}^{n-k}2l(\delta_{l}-\gamma_{l})\right)\cdot D.
		\end{equation}
		Thus, in order to check the equality $[D_h,D_e]=2D_e$ it suffices to check that for each term of $D_e$ the value of $\sum\limits_{j=1}^{k}2j(\beta_{j}-\alpha_{j})+\sum\limits_{l=1}^{n-k}2l(\delta_{l}-\gamma_{l})$ equals $2$ which is clear from the definition of $D_e$ (see \eqref{EFH_diff_formulas}).
		
		\textbf{Step 3.} We prove that $D_e$, $D_f$ and $D_h$ coincide with operators $E$, $F$ and $H$ defined in Subsection 3.2. We will check this for $D_f$ and $D_h$ first. According to Proposition \ref{EFH_schubert} the action of $F$ corresponds to the multiplication by the negative of the first Chern class, i.e. by $-c_1(S)=-p_1$. Hence, the action of $F$ coincides with the action of $D_f=-p_1$. 
		
		In order to prove that operators $D_h$ and $H$ coincide it suffices to check that the image of any element $a\in\mathcal{A}$ of the degree $\deg(a)=d$ under the differential operator $D_h$ equals $H(a)=(-d+k(n-k))\cdot a$. Indeed, for any monomial $a=\prod_{j=1}^{k}p_j^{\alpha_j}\prod_{l=1}^{n-k}q_{l}^{\beta_l}$ we have
		\begin{equation}
			D_h(a)=D_h\left(\prod_{j=1}^{k}p_j^{\alpha_j}\prod_{l=1}^{n-k}q_{l}^{\beta_l}\right)=\left(-2\sum_{j=1}^{k}j\alpha_j-2\sum_{l=1}^{n-k}l\beta_l+k(n-k)\right)\cdot\left(\prod_{j=1}^{k}p_j^{\alpha_j}\prod_{l=1}^{n-k}q_{l}^{\beta_l}\right),
		\end{equation}
		so due to \eqref{degree} we have $D_h(a)=(-d+k(n-k))\cdot a=H(a)$, as desired. 
		
		Finally, in order to prove that $D_e$ and $E$ coincide we use the corollary of the Jacobson-Morozov theorem. Namely, we use the following fact: for given semisimple and nilpotent elements of the semisimple Lie algebra the third element which complements the first two to $\mathfrak{sl}_2$-triple, if exists, is unique. Now note that $\{E,F,H\}$ and $\{D_e,D_f,D_h\}$ are $\mathfrak{sl}_2$-triples. Therefore, equalities $D_h=H$ and $D_f=F$ imply $D_e=E$.
	\end{proof}

	\begin{Rem}
		It is evident that the essential part of the proof consisted of computations. Namely, the proof of Proposition \ref{EFH_operators} reduces to the computation of commutators $[D_e,R_m]$, $[D_f,R_m]$, $[D_h,R_m]$, $[D_e,D_f]$, $[D_h,D_e]$ and $[D_h,D_f]$.
	\end{Rem}

	\section{Construction of the operator $D_e$}
	
	In this section we outline the way in which we obtained the expression for the operator $D_e$. From the formal point of view, there this is unnecessary for the proof of Proposition \ref{EFH_operators} since we have already shown that $D_e$ corresponds to the action of $E$ on cohomology. However, it might be beneficial to describe the general approach which might be useful for similar problems.
	
	Suppose that we have a linear operator $L$ on polynomial algebra $\mathbb{C}[x_1,\ldots,x_m]$. We will show an algorithm which represents $T$ as a differential operator, i.e. in form
	\begin{equation}
		L=P\left(\frac{\partial}{\partial x_1},\ldots,\frac{\partial}{\partial x_m}\right),
	\end{equation}
	where $P(y_1,\ldots,y_m)$ is a polynomial in $m$ variables with coefficients in $\mathbb{C}[x_1,\ldots,x_m]$:
	\begin{equation}
		P(y_1,\ldots,y_m)=\sum_{k_1,\ldots,k_m}p_{k_1,\ldots,k_m}(x_1,\ldots,x_m)\prod_{j=1}^{m}y_j^{k_j}.
	\end{equation}
	In other words, our goal is to express $T$ in the form
	\begin{equation}
		L=\sum_{k_1,\ldots,k_m}p_{k_1,\ldots,k_m}(x_1,\ldots,x_m)\frac{\partial^{k_1+\ldots+k_m}}{\partial x_1^{k_1}\ldots\partial x_m^{k_m}}
	\end{equation}
	for some polynomials $p_{k_1,\ldots,k_m}\in\mathbb{C}[x_1,\ldots,x_m]$.
	\begin{Def}
		Let $\mathcal{S}$ be a complex commutative algebra with unit. An operator $L\colon\mathcal{S}\to\mathcal{S}$ is  called a differential operator of order 0 if $L=L_a$ (or simply $L=a$) for some $a\in S$ where $L_a$ is an operator of multiplication by $a$: $L_a(b)=a\cdot b$.
		An operator $L\colon\mathcal{S}\to\mathcal{S}$ is  called a differential operator of order $d\ge 1$ if for any $a\in S$ the commutator $[L,L_a]=[L,a]$ is a differential operator of order $d-1$ (or less). (Here we consider elements of $\mathcal{S}$ as corresponding scalar operators.)
	\end{Def}
	This definition is motivated by the polynomial algebra: if $\mathcal{S}=\mathbb{C}[x_1,\ldots,x_m]$, then this definition is equivalent to the standard definition of the differential operator of order $d$ in variables $x_1,\ldots,x_m$. 
	\begin{Rem}
		If $\mathcal{S}$ is generated by elements $\{a_j\}_{j\in A}$, then to check that $L$ is a differential operator of order $k$ it is sufficient to check that commutators $[L,a_{j}]$ is a differential operator. In particular, for $\mathcal{S}=\mathbb{C}[x_1,\ldots,x_m]$ we only need to compute commutators of the form $[L,x_j]$.
	\end{Rem}
	For the sake of simplicity we will consider only the case when $\mathcal{S}=\mathbb{C}[x_1,\ldots,x_m]$. However, everything also works in case when $\mathcal{S}$ is a quotient of polynomial a algebra.
	
	The ring of linear operators on $\mathbb{C}[x_1,\ldots,x_m]$ possess a natural structure of the Lie algebra in which the Lie bracket is the commutator $[\cdot,\cdot]$ of two linear operators. We will denote the operator $[\cdot,x]$ as $\Adj_x$.
	Next, note that for any $p(x_1,\ldots,x_m)\in\mathbb{C}[x_1,\ldots,x_m]$ and nonnegative integers $k_1,\ldots,k_m$
	\begin{equation}
		\Adj_{x_j}\left(p(x_1,\ldots,x_m)\frac{\partial^{k_1+\ldots+k_m}}{\partial x_1^{k_1}\ldots\partial x_m^{k_m}}\right)=k_{j}p(x_1,\ldots,x_m)\frac{\partial^{k_1+\ldots+k_m-1}}{\partial x_1^{k_1}\ldots\partial x_{j-1}^{k_{j-1}}\partial x_j^{k_j-1}\partial x_{j+1}^{k_{j+1}}\ldots\partial x_m^{k_m}}.
	\end{equation}
	Thus, for any polynomial $P$ in variables $y_1,\ldots,y_m$ with coefficients in $\mathbb{C}[x_1,\ldots,x_m]$ we have
	\begin{equation}
		\Adj_{x_j}\left(P\left(\frac{\partial}{\partial x_1},\ldots,\frac{\partial}{\partial x_m}\right)\right)=\frac{\partial P}{\partial y_{j}}\left(\frac{\partial}{\partial x_1},\ldots,\frac{\partial}{\partial x_m}\right),
	\end{equation}
	which implies that for any nonnegative integers $k_1,\ldots,k_n$ the following equality holds:
	\begin{equation}\label{adj_action_diff}
		\left(\Adj_{x_1}^{k_1}\circ\ldots\circ\Adj_{x_m}^{k_m}\right)\left(P\left(\frac{\partial}{\partial x_1},\ldots,\frac{\partial}{\partial x_m}\right)\right)=\left(\frac{\partial^{k_1+\ldots+k_m}P}{\partial y_1^{k_1}\ldots\partial y_m^{k_m}}\right)\left(\frac{\partial}{\partial x_1},\ldots,\frac{\partial}{\partial x_m}\right).
	\end{equation}
	The last formula provides a way to compute coefficients of a differential operator if its action on $\mathbb{C}[x_1,\ldots,x_m]$ is given. If $L$ is a differential operator of order $d$, then for any nonnegative numbers $\{k_j\}_{j=1}^{m}$ whose sum grater than $d$, we have
	\begin{equation}
		\left(\Adj_{x_1}^{k_1}\circ\ldots\circ\Adj_{x_m}^{k_m}\right)(L)=0.
	\end{equation}
	Similarly, for any nonnegative $\{k_j\}_{j=1}^{m}$ whose sum equals $d$ we have
	\begin{equation}\label{coeff_formula}
		p_{k_1,\ldots,k_m}(x_1,\ldots,x_m)=\frac{1}{k_{1}!\ldots k_{m}!}\left(\Adj_{x_1}^{k_1}\circ\ldots\circ\Adj_{x_m}^{k_m}\right)(L),
	\end{equation}
	where $p_{k_1,\ldots,k_m}(x_1,\ldots,x_m)$ is the coefficient of $\prod_{j=1}^{m}(\partial/\partial x_j)^{k_j}$ in $L$.
	In this way we can find the terms of order $d$ in $L$ if we know the action of $L$ on $\mathbb{C}[x_1,\ldots,x_m]$. After that we can subtract these terms from $L$ and then repeat this process. Therefore, the algorithm for finding the differential expression for $L$ is as follows:\begin{itemize}
		\item 
		Find the minimal $d\ge 0$ such that for all $j_1,\ldots,j_{j+1}$ all operators $(\Adj_{x_{j_1}}\circ\ldots\circ\Adj_{x_{j_{d+1}}})(L)$ are zero. If $L$ is indeed a differential operator, then such $d$ exists and equals to the order of $L$.
		
		\item
		Using formula \eqref{coeff_formula} find all terms of the (highest) order $d$ in $L$. After that subtract all these terms from $L$ and denote the resulting operator as $L'$.
		
		\item
		Now $L'$ is a differential operator whose order is at most $d-1$. Repeat the second part of the algorithm for $L'$.
		
		\item
		After $d$ steps we will find the differential expression for $L$ which is the sum of all terms found in the second part of the algorithm.
	\end{itemize}
	\begin{Rem}
		It is possible to find a closed formula for $L$ in terms of commutators $(\Adj_{x_{j_1}}\circ\ldots\circ\Adj_{x_{j_{n}}})(L)$ but it seems that such formula would be quite complicated and not suitable for any direct computations.
	\end{Rem}
	
	In our particular problem we knew the action of the operator $E$ on Schubert classes which form a basis of $H^*(\Gr(k,n),\mathbb{C})$. With the help of computer it was found out that for all $i$ and $j$ the operator $(\Adj_{x_{i}}\circ\Adj_{x_{j}})(E)$ is zero. The latter means that $E$ is an differential operator of the order at most 2.
	After that it remained to compute coefficients via the formula \eqref{coeff_formula}. In order to do that we also needed to express Schubert classes in terms of Chern classes. This was done via the correspondence from the Proposition \ref{symm_func_isom}, formula \eqref{chern_schubert} and classical formulas from the theory of symmetric functions.

	\section{Concluding remarks}
	
	In this section we give final remarks and discuss several questions related to our work which remain unanswered in this paper but still seem to be important. 
	
	Our description of $\mathrm{GL}_n$-structure on the cohomology ring of the complex Grassmannian is quite complicated and uses a nontrivial isomorphism. This raises the following question.
	\begin{Ques}
		Is it possible to describe the $\mathrm{GL}_n$-structure on $H^*(\Gr(k,n),\mathbb{C})$ directly, i.e without reference to the Schubert classes?
	\end{Ques}
	\noindent
	However, we have already seen that even the expression for $E$ is fairly complex, so it is unlikely that $\mathrm{GL}_n$-action on cohomology has a simple description. 
	
	Next, since $H$ and $F$ can be described in terms of cohomology it is natural to expect that the same is true for $E$. Nevertheless, it is not clear how the differential expression of $E$ can be translated to the language of cohomology.
	\begin{Ques}
		Does the operator $E$ have a geometrical meaning similar to $H$ and $F$ (see \eqref{HF_cohomology})?
	\end{Ques}
	\noindent
	We also could consider the action of $\{E,F,H\}$ on Schubert classes. Proposition \ref{EFH_schubert} suggests that action of $F$ on $\sigma_{\lambda}$ corresponds to the addition of one box to $\lambda$ in all possible ways. Similarly, action of $E$ corresponds to the deletion of boxes from $\lambda$ but the meaning of the coefficients $(k+\lambda_l-l)(n-k-\lambda_l+l)$ is unclear. It would be interesting to find a combinatorial interpretation of the action of $E$ on Schubert classes.
	
	Finally, our algebra $\mathcal{A}$ arose as the cohomology ring of the compact K\"{a}hler manifold (complex Grassmannian in our case). Such algebras possess special properties because of the Poincare duality and the hard Lefschetz theorems. There are similar objects in commutative algebra called the Artinian Gorenstein algebras (see \cite[Section 2]{Maeno-Watanabe} for the details). One particular class of examples of Artinian Gorenstein algebras consists of the quotients $\mathbb{C}[x_1,\ldots,x_n]/I$, where the ideal $I$ is generated by some polynomial $P\in\mathbb{C}[y_1,\ldots,y_n]$ in the following way
	\begin{equation}
		I=\operatorname{Ann} P,~\text{where}~\operatorname{Ann} P=\left\{f\in\mathbb{C}[x_1,\ldots,x_n]\colon f\left(\frac{\partial}{\partial x_1},\ldots,\frac{\partial}{\partial x_n}\right)P=0\right\}.
	\end{equation} 
	In this case the polynomial $P$ is called the cogenerator of the initial algebra (see also \cite[Theorem 2.1]{Maeno-Watanabe}). It can be checked that for our algebra $\mathcal{A}$ the cogenerator equals
	\begin{equation}
		P(x_1,\ldots,x_k,y_1,\ldots,y_{n-k})=\sum_{\alpha,\beta}\frac{1}{\alpha_1!\,\ldots\,\alpha_k!\,\beta_1!\,\ldots\,\beta_{n-k}!}\left(\prod_{j=1}^{k}c_j(S)^{\alpha_j}\prod_{l=1}^{n-k}c_{l}(Q)^{\beta_{l}}\right)\cdot x_1^{\alpha_1}\ldots x_k^{\alpha_k}\,y_1^{\beta_1}\ldots y_{n-k}^{\beta_{n-k}},
	\end{equation}
	where summation is over all nonnegative integers $\alpha_j$ and $\beta_l$ that satisfy 
	\begin{equation}
		\alpha_1+2\alpha_2+\ldots+k\alpha_k+\beta_1+2\beta_2+\ldots+(n-k)\beta_{n-k}=k(n-k).
	\end{equation}
	Thus, the ideal $I$ generated by elements $R_m\in\mathbb{C}[p_1,\ldots,p_k,q_1,\ldots,q_{n-k}]$ can be also represented as
	\begin{equation}
		I=\operatorname{Ann} P=\left\{f\in\mathbb{C}[p_1,\ldots,p_k,q_1,\ldots,q_{n-k}]\colon f\left(\frac{\partial}{\partial p_1},\ldots,\frac{\partial}{\partial p_k},\frac{\partial}{\partial q_1},\ldots\frac{\partial}{\partial q_{n-k}}\right)P=0\right\}.
	\end{equation}
	The proof of this fact is similar to the proof of the Theorem 2.1 in \cite{Kaveh}.
	\begin{Rem}
		Here we identify the one-dimensional top cohomology group $H^{k(n-k)}(\Gr(k,n),\mathbb{C})$ with $\mathbb{C}$. Hence, $P$ is defined up to nonzero scalar.
	\end{Rem}

	It appears that closed form for this polynomial $P$ is unknown as well as its connection to the complex Grassmannian, so we have the following question. 
	\begin{Ques}
		How does the polynomial $P$ relate to the complex Grassmannian?
	\end{Ques}
	The cogenerators are also related to the so called \textbf{volume polynomials}, see \cite[Section 4]{Kaveh}. We suspect that $P$ might be related to the hypersymplex $\Delta_{kn}$ which is the moment polytope of the Grassmannian with respect to the action of the maximal torus. Specifically, it seems that the cogenerator of the cohomology ring of $\Gr(k,n)$ coincides with the volume polynomial of the hypersimplex.

\end{document}